\documentclass[twoside,11pt,reqno]{amsart}
\usepackage[utf8]{inputenc}
\usepackage[T1]{fontenc}

\usepackage[normalem]{ulem}
\usepackage{mathtools}
\usepackage{amsmath,amssymb,amsthm}
\usepackage{amsfonts}
\usepackage{graphicx,tikz}
\usepackage{setspace}
\usepackage{hyperref}
\usepackage[alphabetic,initials]{amsrefs}
\usepackage{mathtools}
\mathtoolsset{showonlyrefs=true}
\usepackage{bbm}

\newtheorem{thm}{Theorem}[section]
\newtheorem{lem}[thm]{Lemma}

\newtheorem{prop}[thm]{Proposition}

\newcommand{\cF}{\mathcal{F}}

\newcommand{\cC}{\mathcal{C}}

\newcommand{\A}{\mathcal{A}}

\newcommand{\CC}{\mathbb{C}}
\newcommand{\ZZ}{\mathbb{Z}}

\DeclareMathOperator{\Rep}{Rep}
\DeclareMathOperator{\Hom}{Hom}

\DeclareMathOperator{\id}{id}

\DeclareMathOperator{\tr}{tr}

\DeclareMathOperator{\Aut}{Aut}

\usepackage{xspace}
\DeclareRobustCommand{\eg}{e.g.\@\xspace}
\DeclareRobustCommand{\cf}{cf.\@\xspace}
\DeclareRobustCommand{\ie}{i.e.\@\xspace}
\makeatletter
\DeclareRobustCommand{\etc}{%
    \@ifnextchar{.}%
        {etc}%
        {etc.\@\xspace}%
}
\makeatother

\DeclareMathOperator{\Irr}{Irr}
\newcommand{\tRep}[1]{#1\text{--}\Rep}
\newcommand{\bim}[4][]{{}\prescript{\vphantom{#1}}{#2}{#3}^{#1}_{#4}}

\DeclareMathOperator{\Rank}{rk}
\newcommand{\rk}[1]{\Rank(#1)}
\newcommand{\tu}{\mathbbm{1}} %

\begin{document}
\date{\today}
\dateposted{\today}
\title[The rank of $G$-crossed braided extensions of MTCs]{The rank of $G$-crossed braided extensions of modular 
tensor categories}

\address{Department of Mathematics, Morton Hall 321, 1 Ohio University, Athens, OH 45701, USA}
\author{Marcel Bischoff}
\email{bischoff@ohio.edu}
\email{marcel@localconformal.net}

\thanks{Supported in part by NSF Grant DMS-1700192/1821162}

\begin{abstract}
  We give a short proof for a well-known formula for the rank of 
  a $G$-crossed braided extension of a modular tensor category.
\end{abstract}
\maketitle

\section{Introduction}
$G$-crossed braided extensions of modular tensor categories 
are an important ingredient for the process of gauging in
topological phases of matter in the framework of modular tensor categories
\cite{BaBoChWa2014}.
Let $\cC$ be a modular tensor category, which for the purpose of this note is a semisimple, $\CC$-linear abelian ribbon category with simple tensor unit, 
such that the set of isomorphism classes of simple objects $\Irr(\cC)$ is finite and the braiding is non-degenerate, see \eg \cite{BaKi2001}.
A global symmetry \cite{BaBoChWa2014} of $\cC$
is a pair $(G,\rho)$ consisting of a finite group $G$ and homomorphism $\rho\colon G\to \Aut^\mathrm{br}_\otimes(\cC)$, where $\Aut^\mathrm{br}_\otimes(\cC)$
is the group of isomorphism classes of braided autoequivalences of $\cC$.
Gauging a global symmetry $(G,\rho)$ of $\cC$ \cite{CuGaPlZh2016} is a two step process which eventually produces a new modular tensor category
and is given as follows:
\begin{enumerate}
  \item
Construct a $G$-crossed braided extension $\cF=\bigoplus_{g\in G} \cF_g$ 
of $\cC$.
\item Consider the equivariantization $\cF^G$ of $\cF$ to obtain a new modular tensor category, which therefore contains $\Rep(G)$ as a symmetric subcategory.
\end{enumerate}
A $G$-crossed braided fusion category is a fusion category $\cF$
equipped with the following structure:
\begin{itemize}
  \item $\cF$ is faithfully $G$-graded, \ie $\cF=\bigoplus \cF_g$
    and $X_g\otimes Y_h\in\cF_{gh}$ for every $g,h\in G$,
    $X_g\in\cF_g$, and $Y_h\in\cF_h$.
  \item There is an $G$-action $X\mapsto \prescript g{}{\!X}=\rho(g)(X)$ 
    given by a monoidal functor $\underline{\rho}\colon \underline{G}\to 
    \underline{\Aut_\otimes} (\cF)$, where $\underline{\Aut_\otimes} (\cF)$
  is the categorical group of monoidal autoequivalences.
  \item There is a natural family of isomorphisms
    \begin{align}
      c_{X,Y} &\colon X\otimes Y\to \prescript g{}Y\otimes X\,, 
      &g\in G, X\in\cF_g, Y\in \cF
    \end{align}
    with $\rho(g)(\cF_h)\subset \cF_{ghg^{-1}}$ 
    fulfilling certain coeherence diagrams, see \eg \cite{Tu2010}.
\end{itemize}
We note that $\cF_e$ is a braided fusion category and 
we get  a global symmetry $\rho\colon G\to \Aut^\mathrm{br}_\otimes(\cF_e)$ given by  restriction and  truncation of $\underline{\rho}$. 
We are only interested in the case where $\cF_e$ is a modular tensor category.
We refer to \cite{DrGeNiOs2010,Tu2010} for more details.

We call a $G$-crossed braided fusion category $\cF=\bigoplus_{g\in G}\cF_g$ a
$G$-crossed braided extension of a modular tensor category $\cC$ if $\cF_e$ is 
braided equivalent to $\cC$.
The above data of $\cF$ gives by restriction a global symmetry $(G,\rho)$ on 
$\cC\cong\cF_e$.
We note that given a modular tensor category $\cC$ 
and a global action $(G,\rho)$ as an input one has to 
check that certain obstructions have to vanish 
in order for a $G$-crossed braided
to exist, see \cite{EtNiOs2010}.

$G$-crossed braided extensions arise 
naturally in rational conformal field theory. 
Let $\A$ be a (completely) rational conformal net, then 
the category of representations $\Rep(\A)$ is a unitary modular tensor category
\cite{KaLoMg2001}.
Let $G\leq \Aut(\A)$ be a finite group of automorphisms of the net $\A$.
There is a category $\tRep G(\A)$ of $G$-twisted representations of $\A$,
which is a $G$-crossed braided extension of the unitary modular tensor category $\Rep(\A)$.
The fixed point or ``orbifold net'' $\A^G$ is again completely rational and 
$\Rep(\A^G)$ is a gauging of $\Rep(\A)$ by $G$, \ie $\Rep(\A^G)$ is braided equivalent to the $G$-equivariantization 
$(\tRep G(\A))^G$, see \cite{Mg2005} for the original reference and \cite{Bi2018} for a review.

Starting with a modular tensor category $\cC$ 
we want to compute certain invariants of possible 
$G$-crossed braided extensions  $\cF=\bigoplus_{g\in G} \cF_g$ with 
$\cF_e\cong \cC$.
The simplest invariant of $\cF$ is its rank $\rk{\cF}:=|\Irr(\cF)|$.

We note that the global symmetry $(G,\rho)$ associated with 
a $G$-crossed braided extension $\cF$ of $\cC$ equips $\Irr(\cC)$ 
with a $G$-action.
This $G$-space already contains all the information about the rank of 
possible $G$-crossed braided extensions associated with $(G,\rho)$.
Namely, the following well-know formula, \cf \cite[Eq.\ (348)]{BaBoChWa2014}
holds.
\begin{prop}
  \label{prop:main}
  Let $\cF=\bigoplus_{g\in G} \cF_g$ be a (faithful) $G$-crossed 
  braided extension of a modular tensor category $\cC\cong \cF_e$.
  Then the rank  of $\cF_g$ is equal to the size  of the stabilizer $\Irr(\cC)^g$, \ie
  \begin{align}
    \rk{\cF_g}&=|\Irr(\cC)^g|\,.
  \intertext{In particular, the
  rank of $\cF$ is 
determined via the $G$-action on $\Irr(\cC)$ by}
    \rk{\cF}&=\sum_{g\in G}|\Irr(\cC)^g|=|G| |\Irr(\cC)/G|\,.
  \end{align}
\end{prop}
The goal of this note is to provide a short proof of this statement using modular 
invariants, which were used in the operator algebra literature,
\eg \cite{BcEvKa1999,BcEvKa2000}.

\section{The rank of module categories and $G$-crossed braided extensions}
Let $\cC$ be a modular tensor category and $A$ a simple non-degenerate algebra 
object in $\cC$. 
Here non-degenerate means that the trace pairing 
$\Phi_A\colon A \to \overline{A}$ given 
by \[\Phi_A=((\mathrm{ev}_{ A})\circ(\id_{\bar A}\otimes m) \otimes \id_{\bar A})
\circ(\id_{\bar A} \otimes (\id_{A}\otimes\mathrm{coev}_A)\circ m)\circ(\mathrm{coev}_{\bar A} \otimes \id_A)\]
is invertible,
which implies that $A$ has the structure of a special symmetric 
Frobenius algebra object, see \cite[Lemma 2.3]{KoRu2008}.

We denote the category with the opposite braiding by $\overline{\cC}$.
The full center $Z(A)$ is a Lagrangian algebra object in $Z(\cC)\cong\cC\boxtimes\overline{\cC}$ \cite{FrFuRuSc2006,KoRu2008,DaMgNiOs2013}. 
We note that the forgetful functor $F\colon \cC\boxtimes \overline{\cC}\to \cC$ 
maps $X\boxtimes Y\mapsto X\otimes Y$ and has an adjoint 
$I\colon \cC \to \cC\boxtimes \overline{\cC}$. 
The image $I(\tu)$ has the canonical structure of a Lagrangian algebra
in $\cC\boxtimes \overline{\cC}$. If $X\in \cC$ we denote by $\bar X\in \cC$ 
a dual object. 
The full center $Z(A)$  can be realized  as the left center  of the braided product of $(A\boxtimes\tu) \otimes^+ I(\tu)$ \cite{FrFuRuSc2006,KoRu2008}. 
Given a simple non-degenerate algebra $A$ there are two functors 
$\alpha^\pm$ from $\cC$ to the category $\bim A\cC A$ of $A$-bimodules in $\cC$.
They are given by equipping the image of the free right module functor $-\otimes A$
with a left action using the braiding or opposite braiding,
respectively, see
\cite{FrFuRuSc2006}.

The following proposition is well-known to experts.
\begin{prop} 
  \label{prop:ModuleRank}
  Let $\cC$ be a modular tensor category, 
  $A$ a simple non-degenerate algebra object in $\cC$,
  and $Z$ the associated modular invariant matrix, \ie 
  $Z=(Z_{X,Y})_{X,Y\in \Irr(\cC)}$ is the square matrix given by
  \begin{align}
    Z_{X,Y}&=\dim_\CC \Hom_{\!\!\bim A\cC A}(\alpha^+(X),\alpha^-(\bar Y))\,.
  \end{align}
  Then
  \begin{enumerate}
    \item
      The full center $Z(A)$ of $A$   as an object in 
      $\cC\boxtimes \overline{\cC}$ is given by
      \begin{align}
        Z(A)&\cong \bigoplus_{X,Y\in \Irr(\cC)} Z_{X,Y} \cdot X\boxtimes \bar Y\,.
      \end{align}
     \item The rank $\rk{\cC_A}$ of the category of right $A$-modules
      $\cC_A$ is given by the trace of the matrix $Z$: 
      \begin{align}
        \rk{\cC_A} &= \tr(Z)=\sum_{X\in \Irr(\cC)} Z_{X,X}\,.
      \end{align}
    \end{enumerate}
\end{prop}
\begin{proof}
  The first statement is \cite[Remark 3.7]{FrFuRuSc2006}.
  The second statement is 
  \cite[Corollary 6.1]{BcEvKa1999} in the case that $A$ is a Q-system 
  and $\cC$ is a unitary modular tensor category realized as endomorphisms of 
  a type III factor. 
  In the more general setting, it 
  follows  from  \cite[Eq.\ (4.4) and Prop.\ 4.3]{KoRu2008}.
  Namely, $F(Z(A))\cong \bigoplus_{X,Y\in\Irr(\cC)} Z_{X,Y}X\otimes \bar Y$
  thus $\dim\Hom(\tu,F(Z(A))=\tr(Z)$ equals the number of isomorphism classes
  of simple modules in $\cC_A$, since $F(Z(A))$ 
  is equivalent to $\bigoplus_{M\in\Irr(\cC_A)} M\otimes_A \bar M$
  and since $M\otimes_A\bar M$ is a connected algebra by \cite[Proof of Prop.\ 4.10]{KoRu2008}.
\end{proof}
For $\phi\in \Aut_{\otimes}^{\mathrm{br}}(\cC)$
there is a canonical Lagrangian algebra 
$L_\phi=
(\id \boxtimes \phi)I(1)$  in $\cC\boxtimes \overline{\cC}$,
see \eg \cite[Sec 3.2-3.3]{DaNiOs2013}.
The algebra $L_\phi$ corresponds to the full center of a $\cC$-module category,
which we denote by $\cC_\phi$, \cf \cite{KoRu2008,DaMgNiOs2013}.
This module category can be obtained, using \eg \cite{KoRu2008},  by taking any connected summand $A$ of the algebra $F(L_\phi)$ 
(which  always fulfills $Z(A)=L_\phi$) and 
 considering the $\cC$-module category
$\cC_\phi:=\cC_A$, see \cite{KoRu2008},
and also \cite[Lemma 5.1]{EtNiOs2010}.
Since $L_\phi\cong \bigoplus_{X\in\Irr(\cC)} X\boxtimes\phi(\bar X)$, \ie
$Z_{X,Y}=\delta_{\phi(X),Y}$, Proposition 
  \ref{prop:ModuleRank}
immediately gives the following statement.
\begin{lem}
  \label{lem}
 The rank of $\cC_\phi$ equals $|\Irr(\cC)^\phi|$, \ie the number
  of isomorphism classes of simple objects in $\cC$ which are fixed by $\phi$.
\end{lem}

\begin{proof}[Proof of Proposition \ref{prop:main}]
  Let $\cF=\bigoplus_{g\in G} \cF_g$ be a $G$-crossed braided extension
  of a modular tensor category $\cC$. 
  By \cite[Sec.\ 5.4]{EtNiOs2010} the category $\cF_g$ is equivalent as 
  a left $\cC$-module 
  category to the category $\cC_g$ in Lemma \ref{lem}  
  and the statement follows.
\end{proof}

\subsection{The rank of permutation extensions}
  Let $\cC$ be a modular tensor category.
  Gannon and Jones announced \cite{GaJo2018}
  that certain cohomological obstructions
  vanish and thus that there is always an 
  $S_n$-crossed braided extension of $\cC^{\boxtimes n}$
  which we denote by $\cC\wr S_n$, where the action of $S_n$ is given 
  by permuations of objects.
  More general, for any subgroup $G\leq S_n$
  there is a $G$-crossed braided extension $\cC\wr G$.

  For the cyclic subgroup $\ZZ_n\cong \langle (12\cdots n)\rangle \subset S_n$
  with $n$ prime we obtain
  \begin{align}
    \rk{\cC\wr \ZZ_n}=\rk{\cC}^n+(n-1)\rk{\cC}\,,
\end{align}
which in the case $n=2$  
is also true by construction for the examples considered in \cite{EdJoPl2018}.

In general, for the full symmetric group, we get that $\rk{(\cC\wr S_n)_g}$ 
depends only on the conjugacy class $C_g$ of $g\in S_n$.
Namely, $\rk{\cC\wr S_n)_g}=\rk{\cC)}^{|a|}$  where $a=(a_1,\ldots, a_n)$ and 
  $a_j$ is the number 
  of $j$-cycles (counting 1-cycles) in the cycle decomposition of $g$ and $|a|=\sum_ja_j$.
  Thus summing over all elements of the $p(n)$ conjugacy classes we get
  \begin{align}
    \rk{\cC\wr S_n}&=
    \sum_{a=(a_1,\ldots,a_n)}
    c_a \rk{\cC}^{|a|}
    \,,&
    c_a=\frac{n!}{\prod_j(j)^{a_j}(a_j!)}\,,
  \end{align}
  where $c_a=|C_g|$ the size of the conjugacy class $C_g$ of an element $g\in S_n$ with 
  cycle type $a=(a_1,\ldots,a_n)$ and the sum runs over all unordered 
  partitions of $\{1,\ldots,n\}$ with $a_j$ the number of partitions of length $j$. 
  Explicitly,
  \begin{align}
    \rk{\cC\wr S_3}&=\rk{\cC}^3+3\rk{\cC}^2+2\rk{\cC}\\
    \rk{\cC\wr S_4}&=\rk{\cC}^4+6\rk{\cC}^3+11\rk{\cC}^2+6\rk{\cC}\\
    \rk{\cC\wr S_5}&=
    \rk{\cC}^5+10\rk{\cC}^4+35\rk{\cC}^3+50\rk{\cC}^2+24\rk{\cC}\\\cdots
  \end{align}

\subsection*{Acknowledgments} 
I would like to thank Corey Jones, David Penneys, and Julia Plavnik for discussions.

\def\cprime{$'$}\newcommand{\noopsort}[1]{}

\address

\begin{bibdiv}
\begin{biblist}

\bib{BaBoChWa2014}{article}{
      author={Barkeshli, Maissam},
      author={Bonderson, Parsa},
      author={Cheng, Meng},
      author={Wang, Zhenghan},
       title={Symmetry, defects, and gauging of topological phases},
        date={2014},
     journal={arXiv preprint arXiv:1410.4540},
}

\bib{BcEvKa2000}{article}{
      author={Böckenhauer, Jens},
      author={Evans, David~E.},
      author={Kawahigashi, Yasuyuki},
       title={{Chiral structure of modular invariants for subfactors}},
        date={2000},
        ISSN={0010-3616},
     journal={Comm. Math. Phys.},
      volume={210},
      number={3},
       pages={733–784},
         url={http://dx.doi.org/10.1007/s002200050798},
      review={\MR{1777347 (2001k:46097)}},
}

\bib{BcEvKa1999}{article}{
      author={Böckenhauer, Jens},
      author={Evans, David~E.},
      author={Kawahigashi, Yasuyuki},
       title={{On {$\alpha$}-induction, chiral generators and modular
  invariants for subfactors}},
        date={1999},
        ISSN={0010-3616},
     journal={Comm. Math. Phys.},
      volume={208},
      number={2},
       pages={429–487},
         url={http://dx.doi.org/10.1007/s002200050765},
      review={\MR{1729094 (2001c:81180)}},
}

\bib{Bi2018}{article}{
      author={Bischoff, Marcel},
       title={Conformal net realizability of {T}ambara-{Y}amagami categories
  and generalized metaplectic modular categories},
        date={2018},
     journal={arXiv preprint arXiv:1803.04949},
}

\bib{BaKi2001}{book}{
      author={Bakalov, Bojko},
      author={Kirillov, Alexander, Jr.},
       title={Lectures on tensor categories and modular functors},
      series={University Lecture Series},
   publisher={American Mathematical Society, Providence, RI},
        date={2001},
      volume={21},
        ISBN={0-8218-2686-7},
      review={\MR{1797619}},
}

\bib{CuGaPlZh2016}{article}{
      author={Cui, Shawn~X.},
      author={Galindo, C\'esar},
      author={Plavnik, Julia~Yael},
      author={Wang, Zhenghan},
       title={On gauging symmetry of modular categories},
        date={2016},
        ISSN={0010-3616},
     journal={Comm. Math. Phys.},
      volume={348},
      number={3},
       pages={1043\ndash 1064},
         url={https://doi.org/10.1007/s00220-016-2633-8},
      review={\MR{3555361}},
}

\bib{DrGeNiOs2010}{article}{
      author={Drinfeld, Vladimir},
      author={Gelaki, Shlomo},
      author={Nikshych, Dmitri},
      author={Ostrik, Victor},
       title={On braided fusion categories. {I}},
        date={2010},
        ISSN={1022-1824},
     journal={Selecta Math. (N.S.)},
      volume={16},
      number={1},
       pages={1\ndash 119},
         url={http://dx.doi.org/10.1007/s00029-010-0017-z},
      review={\MR{2609644 (2011e:18015)}},
}

\bib{DaMgNiOs2013}{article}{
      author={Davydov, Alexei},
      author={Müger, Michael},
      author={Nikshych, Dmitri},
      author={Ostrik, Victor},
       title={{The {W}itt group of non-degenerate braided fusion categories}},
        date={2013},
        ISSN={0075-4102},
     journal={J. Reine Angew. Math.},
      volume={677},
       pages={135–177},
      review={\MR{3039775}},
}

\bib{DaNiOs2013}{article}{
      author={Davydov, Alexei},
      author={Nikshych, Dmitri},
      author={Ostrik, Victor},
       title={{On the structure of the {W}itt group of braided fusion
  categories}},
        date={2013},
        ISSN={1022-1824},
     journal={Selecta Math. (N.S.)},
      volume={19},
      number={1},
       pages={237–269},
         url={http://dx.doi.org/10.1007/s00029-012-0093-3},
      review={\MR{3022755}},
}

\bib{EdJoPl2018}{article}{
      author={Edie-Michell, Cain},
      author={Jones, Corey},
      author={Plavnik, Julia},
       title={Fusion rules for {$\mathbb{Z}/ 2\mathbb{Z}$} permutation
  gauging},
        date={2018},
     journal={arXiv preprint arXiv:1804.01657},
}

\bib{EtNiOs2010}{article}{
      author={Etingof, Pavel},
      author={Nikshych, Dmitri},
      author={Ostrik, Victor},
       title={Fusion categories and homotopy theory},
        date={2010},
        ISSN={1663-487X},
     journal={Quantum Topol.},
      volume={1},
      number={3},
       pages={209\ndash 273},
         url={http://dx.doi.org/10.4171/QT/6},
        note={With an appendix by Ehud Meir},
      review={\MR{2677836}},
}

\bib{FrFuRuSc2006}{article}{
      author={Fröhlich, Jürg},
      author={Fuchs, Jürgen},
      author={Runkel, Ingo},
      author={Schweigert, Christoph},
       title={{Correspondences of ribbon categories}},
        date={2006},
        ISSN={0001-8708},
     journal={Adv. Math.},
      volume={199},
      number={1},
       pages={192–329},
         url={http://dx.doi.org/10.1016/j.aim.2005.04.007},
      review={\MR{2187404 (2007b:18007)}},
}

\bib{GaJo2018}{article}{
      author={Gannon, Terry},
      author={Jones, Corey},
       title={Vanishing of categorical obstructions for permutation orbifolds},
        date={2018},
     journal={arXiv preprint arXiv:1804.08343},
}

\bib{KaLoMg2001}{article}{
      author={Kawahigashi, Y.},
      author={Longo, Roberto},
      author={Müger, Michael},
       title={{Multi-Interval Subfactors and Modularity of Representations in
  Conformal Field Theory}},
        date={2001},
     journal={Comm. Math. Phys.},
      volume={219},
       pages={631–669},
      eprint={arXiv:math/9903104},
}

\bib{KoRu2008}{article}{
      author={Kong, Liang},
      author={Runkel, Ingo},
       title={{Morita classes of algebras in modular tensor categories}},
        date={2008},
        ISSN={0001-8708},
     journal={Adv. Math.},
      volume={219},
      number={5},
       pages={1548–1576},
         url={http://dx.doi.org/10.1016/j.aim.2008.07.004},
      review={\MR{2458146 (2009h:18016)}},
}

\bib{Mg2005}{article}{
      author={Müger, Michael},
       title={{Conformal Orbifold Theories and Braided Crossed G-Categories}},
        date={2005},
        ISSN={0010-3616},
     journal={Comm. Math. Phys.},
      volume={260},
       pages={727–762},
         url={http://dx.doi.org/10.1007/s00220-005-1291-z},
}

\bib{Tu2010}{book}{
      author={Turaev, Vladimir},
       title={Homotopy quantum field theory},
      series={EMS Tracts in Mathematics},
   publisher={European Mathematical Society (EMS), Z\"urich},
        date={2010},
      volume={10},
        ISBN={978-3-03719-086-9},
         url={http://dx.doi.org/10.4171/086},
        note={Appendix 5 by Michael M\"uger and Appendices 6 and 7 by Alexis
  Virelizier},
      review={\MR{2674592}},
}

\end{biblist}
\end{bibdiv}
\end{document}